\documentclass[reqno,11pt]{amsart}
\usepackage{latexsym}
\usepackage{amsfonts}
\usepackage{amssymb}
\usepackage{mathrsfs}
\usepackage[all]{xy}
\usepackage{bbm}
\usepackage{xcolor}
\usepackage{ulem}
\usepackage[mathscr]{euscript}

\usepackage[english]{babel}

\usepackage{amsmath}
\usepackage{graphicx}
\usepackage[colorinlistoftodos]{todonotes}
\usepackage[colorlinks=true, allcolors=blue]{hyperref}
\usepackage{tikz-cd}
\usepackage{tikz}
\usepackage{pst-node}
\usepackage{mathtools}
\usetikzlibrary{graphs}

\usepackage[T2A]{fontenc}

\theoremstyle{plain}
\newtheorem{thm}{Theorem}[section]

\newtheorem{lem}[thm]{Lemma}
\newtheorem{rem}[thm]{Remark}
\newtheorem{prop}[thm]{Proposition}
\newtheorem{cor}[thm]{Corollary}
\newtheorem{defn}[thm]{Definition}

\theoremstyle{definition}

\theoremstyle{remark}

\numberwithin{equation}{section}
\newcommand{\BA}{\mathbf{BA}}
\newcommand{\DI}{\mathbf{DI}}
\newcommand{\SL}{\operatorname{SL}}

\newcommand\mr{M_{m,n}}

\newcommand\amr{$A\in M_{m,n}$}
\newcommand\da{Diophantine approximation}

\newcommand{\nz}{\smallsetminus\{0\}}
\newcommand{\Q}{{\mathbb {Q}}}
\newcommand{\R}{{\mathbb{R}}}

\newcommand{\Z}{{\mathbb{Z}}}
\newcommand{\N}{{\mathbb{N}}}

\newcommand\hd{Hausdorff dimension}
\newcommand{\val}{{\boldsymbol{\alpha}}}
\newcommand{\vbe}{{\boldsymbol{\beta}}}
\newcommand{\vom}{{\boldsymbol{\omega}}}
\newcommand{\va}{{\bf a}}
\newcommand{\vb}{{\bf b}}
\newcommand{\vv}{{\bf v}}
\newcommand{\vr}{{\bf r}}
\newcommand{\vs}{{\bf s}}

\newcommand{\vp}{{\bf p}}
\newcommand{\vq}{{\bf q}}
\newcommand{\x}{{\bf x}}
\newcommand{\vu}{{\bf u}}
\newcommand{\vw}{{\bf w}}
\newcommand{\vk}{{\bf k}}
\newcommand{\vy}{{\bf y}}
\newcommand{\ignore}[1]{{}}

\newcommand {\comm}[1]   {\textcolor{red}{#1}}
\newcommand\eq[2]{
\begin{equation}
\label{eq:#1}
{#2}
\end{equation}
}
\newcommand{\equ}[1]{\eqref{eq:#1}}

\title[Bad versus Dirichlet]{A dichotomy phenomenon for \\ Bad minus normed Dirichlet}
\author{Dmitry Kleinbock}
\author{Anurag Rao}
\address{Brandeis University, Waltham \tt{ kleinboc@brandeis.edu}}
\address{Israel Institute of Technology, Haifa \tt{ rao.anurag@technion.campus.ac.il}}

\begin{document}

\begin{abstract}
    Given a norm $\nu$ on $\mathbb{R}^2$, the set of $\nu$-Dirichlet improvable numbers $\mathbf{DI}_\nu$ was defined and studied in the papers \cite{AD, KR1}. {When $\nu$ is the supremum norm, $\mathbf{DI}_\nu = \mathbf{BA}\cup \Q$, where $\mathbf{BA}$ is the set of badly approximable numbers.}
    Each of the sets $\mathbf{DI}_\nu$, like  $\mathbf{BA}$, is of measure zero and satisfies the winning property of Schmidt. Hence for every norm $\nu$, $\mathbf{BA} \cap \mathbf{DI}_\nu$ is winning and thus has full Hausdorff dimension.     In the present article we prove {the following dichotomy phenomenon:} either $\mathbf{BA} \subset \mathbf{DI}_\nu$ or else $\mathbf{BA} \smallsetminus \mathbf{DI}_\nu$ has full Hausdorff dimension.    We give several examples for each of the two cases. The dichotomy is based on whether the \textit{critical locus} of $\nu$ intersects a precompact $g_t$-orbit,  where $\{g_t\}$ is the one-parameter diagonal subgroup of $\SL_2(\R)$ acting on  the space $X$ of unimodular  lattices in $\mathbb{R}^2$. Thus the aforementioned dichotomy follows from the following dynamical statement: for a lattice $\Lambda\in X$, either $g_\mathbb{R} \Lambda$ is unbounded ({and then any precompact $g_{\mathbb{R}_{>0}}$-orbit must eventually avoid a neighborhood of $\Lambda$}), or not, in which case the set of lattices in $X$ whose  $g_{\mathbb{R}_{>0}}$-trajectories are precompact and contain $\Lambda$ in their closure has full \hd.
\end{abstract}

\thanks{The first-named author was supported by NSF Grant DMS-1900560 and DMS-2155111 and by a grant from the Simons Foundation (922989, Kleinbock). This material is based upon work supported by a grant from the Institute for Advanced Study School of Mathematics.
{The second-named author received funding from the European Research Council (ERC) under the European Union’s Horizon 2020 Research and Innovation Program, Grant agreement no.\ 754475.}}

\maketitle

\section{Introduction}

Dirichlet's theorem (\cite[Theorem 1A]{Schmidt}) states that for every real number $\alpha$ {and every $T > 1$}, the set of inequalities
\begin{equation*}\label{dirichlet}
    \begin{cases}  \langle \alpha q \rangle & \leq T^{-1},   \\
    \ \  |q| & <  T\end{cases}
\end{equation*}
has a solution $q \in \N$. Here $\langle \cdot \rangle$ denotes the distance of a real number {to a nearest integer}.
The question of improving Dirichlet's theorem was initiated by Davenport--Schmidt in \cite{DS}. A real number $\alpha$ is said to be \textit{Dirichlet improvable}, written $\alpha \in \DI_\infty$, if there is a $c<1$ such that the system
\begin{equation}\label{DIinfty-def}
    \begin{cases}  \langle \alpha q \rangle &< cT^{-1},   \\
    \ \ |q|&< T\end{cases}
\end{equation}
can be solved in $q \in \N$ for all sufficiently large $T$. \cite[Theorem 1]{DS} shows that $\DI_\infty$ is exactly the set $\BA \cup \Q$.  {In fact, this was known as far back as the paper of Morimoto \cite{Mor}}.
{We recall that} badly approximable numbers are those real numbers whose continued fraction 
{coefficients (partial quotients)} are bounded from  above. 

A normed variant of Dirichlet's theorem (also referred to as Minkowski's approximation theorem) was studied in \cite{AD, KR1}. The starting point is the observation that the defining condition in \eqref{DIinfty-def} can be restated as
\begin{equation*}
    \Lambda_\alpha \cap \left[ {\begin{array}{cc} cT^{-1} & 0 \\ 0 & T  \end{array}}\right]
    B_\infty(1) \neq \{0\}
\end{equation*}
where $\Lambda_\alpha$ denotes the unimodular lattice $\left[ {\begin{array}{cc} 1 & \alpha \\ 0 & 1 \end{array}}\right]\Z^2$ and $B_\infty(1)$ denotes the supremum norm ball in $\R^2$ centered at the origin and with radius $1$.

{Now, let $X$ denote the set of unimodular lattices in $\R^2$,} fix a norm $\nu$ on $\R^2$ and define the \textit{critical radius} of $\nu$:
\begin{equation}\label{critical-radius}
    r_\nu := \sup\big\{ r \in \R: \text{there exists  }\Lambda \in X \text{ with } \Lambda \cap B_\nu(r) = \{0\}\big\}.
\end{equation}
{By Minkowski's convex body theorem, $r_\nu$ is finite.} Let $B_\nu(r)$ denote  the $\nu$-norm ball with radius $r$ centered at the origin. We say that a real number $\alpha$ is \textit{$\nu$-Dirichlet improvable}, written $\alpha \in \DI_\nu$, if there is some $c<1$ with \begin{equation*}
    \Lambda_\alpha \cap \left[ {\begin{array}{cc} cT^{-1} & 0 \\ 0 & T \end{array}}\right]
    B_\nu\left(r_\nu\right) \neq \{0\}
\end{equation*} 
for all sufficiently large $T$. When $\nu$ is the supremum norm (denoted by the subscript $\infty)$, we have that $r_\infty =1$ and so recover the definition of Davenport--Schmidt.
When $\nu$ is the Euclidean norm (denoted by the subscript $2$), we see that $\alpha \in \DI_2$ if and only if there is some $c<1$ such that
\begin{equation*}
    \left(\frac{T\langle\alpha q \rangle}{c}\right)^2 + \left(\frac{q}{T}\right)^2 < \frac{2}{\sqrt{3}}
\end{equation*}
is solvable {in} $q\in \N$ for all sufficiently large $T$.
The normalizing constant $2/\sqrt{3}$ is equal to 
{the square of the critical radius  of the Euclidean norm}, see e.g.\ \cite{KRS} for a discussion and references. 

Consider the following preliminary properties:
\begin{thm}[\cite{KR1}, Theorems 3.1 and 1.3]\label{DI-winning}
For each norm $\nu$ on $\R^2$, the set $\DI_\nu$ is of measure zero but  winning in the sense of Schmidt.
In particular, $\DI_\nu$ has full Hausdorff dimension. 
\end{thm}
See \cite{S} {and a paragraph before Proposition \ref{uv-winning}} for a discussion of Schmidt games and winning. 
The fact that $\BA$ has full Hausdorff dimension goes back to Jarnik in 1928 (see \cite[Theorem 7.1]{BRV} for an elementary proof). 
Thus, natural questions that come up in the study of Theorem \ref{DI-winning} are
\begin{enumerate}
    \item[(a)] Do there exist norms for which $\BA \not\subset \DI_\nu$?
    \item[(b)] {If yes, can one characterize the norms for which $\BA \not\subset \DI_\nu$?}
\end{enumerate}
The existence of norms which satisfy the above non-containment is a consequence of the equality
\begin{equation*}
    \Q= \bigcap_{\nu} \DI_\nu
\end{equation*}
(see \cite[Theorem 1.7]{KR2}).
In the present paper we address Question (b) by giving a convenient dynamical condition on $\nu$ which completely determines whether or not $\BA$ is a subset of  $\DI_\nu$. 
{In fact, our main results are completely dynamical in nature and this paper is a twofold study of limit points of bounded diagonal orbits in $X$ and criteria to detect whether certain submanifolds intersect precompact orbits.}

We briefly explain how dynamics comes into play; the idea essentially comes from the work of Davenport and Schmidt \cite{DS}. 
Each norm $\nu$ gives rise to the following compact subset of $X$:
\begin{equation*}
    \mathcal{L}_\nu := \big\{\Lambda \in X : \Lambda \cap B_\nu\left(r_\nu\right) = \{0\} \big\}.
\end{equation*}
$\mathcal{L}_\nu$ is referred to as the \textit{critical locus} of the norm. 
It follows from the definition of $r_\nu$ and from Mahler's compactness criterion that $\mathcal{L}_\nu$ is a non-empty compact subset of $X$. 
Equivalently, lattices in $\mathcal{L}_\nu$ give the densest lattice packings of $\R^2$ by the domain $B_\nu\left(\frac{r_\nu}{2}\right)$.
We have the following notation for present and future use:
\begin{equation}\label{1-para-notation}
  {  g_t:= \left[ {\begin{array}{cc} e^t & 0 \\ 0 & e^{-t} \end{array}}\right],
   \text{ where } t
    \in \R,\  \text{ and }g_E
    := \left\lbrace g_t
    : t\in E \right\rbrace \text{ for a subset 
$E$ of $\R$}.}\end{equation}
\begin{prop}[\cite{Da}, Theorem 2.20    and \cite{KR1}, Theorem 2.1]\label{dynamical-restatement}  {We have the following two equivalences}:
\begin{enumerate}
    \item[(i)]  A real number $\alpha$ belongs to $\BA$ if and only if 
    \begin{equation*}
      g_{\R_{>0}}\Lambda_\alpha \text{ is precompact in } X.
    \end{equation*}
    \item[(ii)] Given a norm $\nu$ on $\R^2$, we have that $\alpha \in \DI_\nu$ if and only if
there exists $t_0 >0$ and a neighborhood $U \supset \mathcal{L}_\nu$ {such} that
\begin{equation*}
    g_{\R_{>t_0}} \Lambda_\alpha \cap U = \varnothing.
\end{equation*}
\end{enumerate}
\end{prop}
{With this context}, we state our main theorem:
\begin{thm}\label{main-lemma}
Fix a lattice $\Lambda \in X$. If $g_\R\Lambda$ is precompact in $X$, then the set
\begin{equation*}
    \mathcal{S}_\Lambda:= \left\lbrace \alpha \in \BA : \Lambda \in \overline{g_{\R_{>0}}\Lambda_\alpha}\right\rbrace
\end{equation*}
has Hausdorff dimension $1$. Conversely, if  $g_\R\Lambda$ is not precompact in $X$, then for every $\alpha \in \BA$  there exists  $t_0>0$ and a neighbourhood $U \supset \Lambda$ such that
\begin{equation*}
    g_{\R_{>t_0}}\Lambda_\alpha \cap U = \varnothing.
\end{equation*}
\end{thm}
\begin{rem}\label{thickness-remark} \rm
The converse is elementary and left to the reader as an exercise, {see} \cite[Proposition 4.1]{KR2}. Thus  the main content of the theorem is the first part.
Its proof below can be easily modified to show that the intersection of $\mathcal{S}_\Lambda$ with any open interval has Hausdorff dimension $1$. However, $\mathcal{S}_\Lambda$ is not winning as follows from \cite[Theorem 2.8]{AGK}. (This last theorem was used to prove Theorem \ref{DI-winning}.)
\end{rem}
As a corollary of Proposition \ref{dynamical-restatement} and Theorem \ref{main-lemma}, we have our dichotomy phenomenon:
\begin{thm}\label{dichotomy-thm}
Fix a norm $\nu$ on $\R^2$. If $\mathcal{L}_\nu$ contains a lattice $\Lambda$ {such that }
    $g_\R \Lambda$ is precompact in $X$,
then $\BA \smallsetminus \DI_\nu$ has full Hausdorff dimension.
Conversely, if no such lattice exists, then $\BA \subset \DI_\nu$.
\end{thm}
\begin{rem} \rm
The converse has already been proved in the more general multidimensional setting of weighted approximation of systems of linear forms, {see} \cite[Proposition 4.1]{KR2}.
\end{rem}
It can be shown that the Euclidean critical locus $\mathcal{L}_2$ intersects a precompact $\R$-orbit (in fact, quite a few of them). Moreover, there is a very large class of norms which {also} satisfy this condition. Let us say that a norm $\nu$ on $\R^2$ is 
\textit{irreducible} if whenever $\eta\ne\nu$ is a  norm on $\R^2$ with 
$\eta(\vv)\ge\nu(\vv)$ for any $\vv\in\R^2$, we have $r_\eta > r_\nu$.
Mahler  introduced this distinguished set of norms (or rather, in his terminology, of convex bounded symmetric domains)   in \cite{Ma1}, and they turned out to be central to proving Theorem \ref{DI-winning} among other results.
Examples of irreducible norms are those 
whose unit balls are ellipses, parallelograms and Reinhardt's curvilinear octagon (see \cite[Examples 2.2, 2.3, 3.3]{KRS}).
The critical locus $\mathcal{L}_\nu$ 
of an irreducible norm $\nu$ 
which does not come from a parallelogram is necessarily a one-dimensional $\mathcal{C}^1$-submanifold  of $X$ (see \cite[Theorem 3]{Ma2}). Considerations in the tangent bundle of $X$ then allow us to conclude that $\mathcal{L}_\nu$ must intersect some precompact orbit $g_\R \Lambda$ (see Proposition \ref{irred-transversal} below). Thus, we obtain as a corollary:
\begin{thm}\label{L_2-thm}
If $\nu$ is an irreducible norm on $\R^2$ 
whose unit ball is not a parallelogram, then $\BA \smallsetminus \DI_\nu$ has full Hausdorff dimension. In particular, $\BA \smallsetminus \DI_2$ has full Hausdorff dimension.
\end{thm}
We also mention that, since the critical locus of a hexagonal norm is a singleton (see \cite[Example 2.4]{KRS}), one can easily construct norms $\nu$ which are not irreducible and for which $\BA \smallsetminus \DI_\nu$ has full dimension. For example, if $\nu$ is a norm with $\mathcal{L}_\nu = \{\Lambda_0\}$, one can choose an element $g \in \SL_2(\R)$ which makes $g_\R g\Lambda_0$ precompact, and then notice that $\mathcal{L}_{\nu \circ g^{-1}} = g\mathcal{L}_\nu$.

\smallskip

We organize the paper as follows.
In \S\ref{bounded-orbits} we explore the structure of $X$ more carefully to establish, under the assumption that $\nu$ is a non-parallelogram irreducible norm, the existence of lattices in $\mathcal{L}_\nu$ with precompact $\R$-orbits. With this out of the way, all that remains is to prove Theorem \ref{main-lemma} in the case when the mentioned lattice $\Lambda$ has a precompact $\R$-orbit. To prove it we use a symbolic representation of the diagonal flow. In \S\ref{cross-section} we study a specific Poincare section for $X$ which arises in the exposition of \cite{C2}. 
From this it becomes apparent that the set $\mathcal{S}_\Lambda$ from Theorem \ref{main-lemma} consists of badly approximable {numbers} having certain predetermined blocks of {coefficients} appearing in their continued fraction expansion. We then apply 
a dimension estimate from \cite{U} in \S\ref{hd-computation} to show this set has full Hausdorff dimension. {Some remarks on possible generalizations of the set-up of this paper are made in the last section.}

\section{Critical loci which intersect precompact orbits}\label{bounded-orbits}

We denote $G:= \SL_2(\R)$ and $\Gamma := \SL_2(\Z)$. Identify $X$ with $G/\Gamma$ via the map $g \mapsto g\Z^2$. 
In addition to the notation of \eqref{1-para-notation},
for $x,y,
    \in \R$ denote
\begin{equation*}u_x :=
    \left[ {\begin{array}{cc} 1 & x \\ 0 & 1 \end{array}}\right]\text{ and } v_y := \left[ {\begin{array}{cc} 1 & 0 \\ y  & 1 \end{array}}\right].  
\end{equation*}

We next have the stability of precompact orbits under the action of certain subgroups of $G$:
\begin{prop}\label{commutators}
{For any} $\Lambda \in X$ 
there exists $\varepsilon >0$ such that the maps
\begin{equation}\label{local-coord}
    (-\varepsilon,\varepsilon)^3 \to X;\ (x,t,y) \mapsto v_y g_t u_x\Lambda \ \text{ and } \ (x,t,y) \mapsto u_x g_t v_y \Lambda
\end{equation}
are diffeomorphisms.
Moreover, {for any fixed $(x,t,y) \in \R^3$}, we have the equivalences:
\begin{equation}\label{u-precompact}
    g_{\R_{>0}}\left(v_y g_t u_x \Lambda\right) \text{ is precompact in } X \iff  g_{\R_{>0}}\left( u_x \Lambda\right) \text{ is precompact in } X
\end{equation}
{and}
\begin{equation}\label{v-precompact}
    g_{\R_{<0}}\left(u_x g_t v_y \Lambda\right) \text{ is precompact in } X \iff g_{\R_{<0}} \left(v_y \Lambda \right) \text{ is precompact in } X.
\end{equation}
\end{prop}
\begin{proof}
The first assertion follows from \cite[Theorem 2.10.1]{V}.
For the 
{remaining part of the proposition}, note that  {for any $s \in \R$}
\begin{equation}\label{commutator}
     {g_s v_y g_t u_x \Lambda = v_{ye^{-2s}} g_{s+t} u_x \Lambda\  \text{ and }\  g_s u_x g_t v_y \Lambda = u_{xe^{2s}} g_{s +t} v_y \Lambda.}
\end{equation}
Alternatively, see \cite[Proposition 2.12]{Da}.
\end{proof}
Our main tool in establishing the existence of precompact orbits passing through critical loci is the following proposition which 
{constructs a winning set of} precompact orbits along the $u_x$ and $v_y$ directions about any fixed lattice. {It is a well known application of Schmidt's results in \cite{S}; we give an elementary proof here for convenience (see {also} \cite[Theorem 3.7]{KW} and \cite[Theorem 1.3]{Mc}).} 

For the benefit of the reader we recall the rules of Schmidt's game introduced in \cite{S}. It
involves two parameters $\alpha,\beta\in(0,1)$  and is played by two players Alice and Bob on a complete metric space (which we shall take to be the set of real numbers) with a target set $S$. Bob starts the game by choosing a closed ball $B_0=B(x_0,r_0)$ in $\R$ with center $x_0$ and radius $r_0$. After Bob chooses a closed ball $B_i = B(x_i,{r}_i) $, Alice chooses  $A_i = B(x_i', {r}_i')\subset B_i$ with ${r}'_i=\alpha {r}_i$, 
 and then Bob chooses  $B_{i+1} = B(x_{i+1},{r}_{i+1}) \subset A_i$ with ${r}_{i+1}=\beta {r}'_{i} $, 
etc. Alice wins the game if the unique point  $\bigcap_{i=0}^\infty A_i=\bigcap_{i=0}^\infty B_i$ belongs to $S$, and Bob wins otherwise. The set $S$ is \textit{$(\alpha,\beta)$-winning} if Alice has a winning strategy, 
and is \textit{winning} if it is $(\alpha,\beta)$-winning for some $\alpha > 0$ and all $\beta\in(0,1)$.
\begin{prop}\label{uv-winning}
For any $\Lambda \in X$,
the sets
\begin{equation*}
    \left\lbrace x\in \R: g_{\R_{>0}}u_x\Lambda \text{ is precompact in } X \right\rbrace\  \text{ and } \ 
    \left\lbrace y\in \R: g_{\R_{<0}}v_y\Lambda \text{ is precompact in } X \right\rbrace
\end{equation*}
are winning. 
\end{prop}
\begin{proof} 

{Recall that we are given $\Lambda\in X$.} Let us first 
consider the case when
\begin{equation}\label{aneq0}
    \Lambda = \left[ {\begin{array}{cc} a & 0 \\ c & a^{-1} \end{array}}\right] 
    \left[ {\begin{array}{cc} 1 & b \\ 0 & 1 \end{array}}\right]\Z^2
\end{equation}
for some $a\in\R\nz$ and $b,c\in\R$.
Take $\alpha < 1/2$. After making the first move 
Alice can ensure that  all $x$ in the ball of her choice satisfy $a+cx \neq 0$.
For $\phi(x):= a^{-1}x(a+cx)^{-1}$, we have
\begin{equation*}
    u_x \Lambda =  \left[ {\begin{array}{cc} a+cx& 0 \\ c & (a+cx)^{-1} \end{array}}\right]\
    \left[ {\begin{array}{cc} 1 & \phi(x) + b \\ 0 & 1\end{array}}\right]\Z^2.
\end{equation*}
Using the equivalence of Proposition \ref{commutators} and the characterization of {Proposition} \ref{dynamical-restatement}(i), we see that 
\begin{equation*}
    \phi(x) + b  \in \BA {\ \Longleftrightarrow\ } g_{\R_{>0}} u_x \Lambda \text{ is precompact in } X.
\end{equation*}
Since $\phi^{-1}$ is Lipschitz, the winning property of $\BA$ and \cite[Theorem 1]{S} establishes the winning property for the first set in the proposition.

Next we consider the case when $\Lambda$ cannot be written in the form 
\eqref{aneq0}. This happens if and only if $
    \Lambda = \left[ {\begin{array}{cc} 0 & a \\ b & c \end{array}}\right]  \Z^2$ for some $a,b,c\in\R$. A consequence is that, for any $x \neq 0$, $u_x\Lambda$ can be written in the form 
    \eqref{aneq0}; thus one can replace $\Lambda$ with $u_x\Lambda$  and reduce the problem  to the previously considered case. 
 {Hence} in both cases  we have the winning property for the first set in the proposition; the second set,  in view of an elementary observation that 
\eq{backwards}{ g_{\R_{>0}} u_\alpha \Z^2 \text{ is precompact in } X \ \Longleftrightarrow\ \text{so is }g_{\R_{<0}} v_\alpha \Z^2,}
can be handled similarly.
\end{proof}
We have the notation for the following subgroups of $G$:
\begin{equation*}
F:=\{g_t
: t\in \R\},\ H^+ := \{u_x
: x \in \R\},\ H^- := \{v_y
: y \in \R\}.
\end{equation*}
\begin{rem} \rm
We use the notation $T_m(M)$ to denote the tangent space of a ($\mathcal{C}^1$ or smoother) manifold $M$ at a point $m$. We also use the prefix $d$ to indicate the derivative of a smooth map.
$T_z(Fz)$ is understood to be the image, in $T_z(X)$, of the tangent space $T_e(F)$ via the {derivative of the} map $f\mapsto fz$. The same interpretation is used for the other subgroups of $G$.
\end{rem}
\begin{prop}\label{prop-transversal}
Let $Z\subset X$ be a {one-dimensional} compact $\mathcal{C}^1$-submanifold satisfying, for some $z \in Z$,
\begin{equation}\label{transversality1}
T_z(X) = T_z(Z) + T_z(Fz) + T_z(H^+z)
\end{equation}
and
\begin{equation}\label{transversality2}
    T_z(X)= T_z(Z) + T_z(Fz) + T_z(H^-z).
\end{equation}
Then $Z$ contains a lattice $\Lambda$ for which
$g_\R\Lambda$ is precompact in $X$.
\end{prop}
\begin{proof}[Proof of Proposition \ref{prop-transversal}]
Let 
\begin{equation*}
\phi:(-\varepsilon, \varepsilon) \to X
\end{equation*}
be a local parameterization of $Z$ with $\phi(0)=z$.
Using the first set of local coordinates at $z$ in  \eqref{local-coord}, and the tangent space decomposition of \eqref{transversality2}, we can write
\begin{equation*}
\phi(s)= v_{y(s)}g_{t(s)}u_{x(s)}z
\end{equation*}
where $x(\cdot)$, ${t}(\cdot)$, $y(\cdot)$ are now $\mathcal{C}^1$-functions and $x'(0) \neq 0$.
Thus the inverse function $x^{-1}$ exists in a neigborhood of $0$ and is Lipschitz.
Using Proposition \ref{uv-winning}, \cite[Theorem 1]{S} and  the equivalence \eqref{u-precompact}, we see that
\begin{equation*}
\left\lbrace s \in (-\varepsilon, \varepsilon) : g_{\R_{>0}}\phi(s)  \text{ is precompact in } X\right\rbrace
\end{equation*}
is winning. 
An entirely analogous argument using \eqref{transversality1} shows that
\begin{equation*}
\left\lbrace s \in (-\varepsilon, \varepsilon) : g_{\R_{<0}}\phi(s)  \text{ is precompact in } X\right\rbrace
\end{equation*}
is winning. Since the intersection of winning sets is also winning \cite[Theorem 2]{S}, we arrive at the desired result.
\end{proof}
\begin{prop}\label{irred-transversal}
If $\nu$ is an irreducible norm on $\R^2$ whose unit ball is not a parallelogram, then $\mathcal{L}_\nu$ is a one-dimensional compact $\mathcal{C}^1$-submanifold of $X$ satisfying  conditions  \eqref{transversality1} and \eqref{transversality2} at every point $z\in \mathcal{L}_\nu$. Consequently, $\mathcal{L}_\nu$ must intersect some precompact $\R$-orbit.
\end{prop}
\begin{proof}
The first condition \eqref{transversality1} is proved in \cite[Theorem 3.11]{KR1}.
For the second condition, consider the linear automorphism of $\R^2$ given by
\begin{equation*}
    \vv\mapsto p\vv, \text{ where 
    } p := \left[ {\begin{array}{cc} 0 & 1 \\ 1 & 0 \end{array}}\right].
\end{equation*}
Of course, it also induces a map on subsets of $\R^2$ which gives a diffeomorphism of $X$
\begin{equation*}
    \Lambda \mapsto p\Lambda.
\end{equation*}
{On} the level of critical loci and norms, we have that
\begin{equation*}
    p \mathcal{L}_\nu = \mathcal{L}_{\nu\circ p}.
\end{equation*}
Moreover, if $\nu$ is irreducible with unit ball not a parallelogram, then the same holds true for $\nu \circ p$.
{(Indeed, the irreducibility of   $\nu \circ p$  follows from the definition   after noting that $r_{\nu \circ p} = r_{\nu}$; on the other hand, it is clear that the image of a parallelogram under a linear transformation remains a parallelogram.)}
On  the level of tangent spaces we have that, for $z\in X$,
\begin{equation*}
    (dp)T_z(Fz) = T_{p z}(F pz),\  (dp)T_z(H^+ z) = T_{pz} (H^- pz),\  (dp)T_z(H^- z) = T_{pz}(H^+ pz).
\end{equation*}
The upshot now is that since $\mathcal{L}_{\nu \circ p}$ satisfies \eqref{transversality1} at 
  $pz \in p\mathcal{L}_\nu$, {it follows that} $\mathcal{L}_\nu$ satisfies \eqref{transversality2} at $z$.
  This shows that Proposition \ref{prop-transversal} can be applied, giving the final assertion in the present proposition.
\end{proof}

\section{Minimal vectors and a Poincare section for $X$}\label{cross-section}

We study a well known correspondence between diagonal orbits in $X$ and orbits in an invertible extension of the Gauss map.
We follow the exposition of \cite[\S 3.1]{C2} and presume the reader is familiar with the theory of continued fractions {at} the level of \cite{K}.

If $a, b$ are real numbers, we define the rectangle $R(a,b)$ to be
\begin{equation*}
    R(a,b) := \left\lbrace (x_1,x_2)\in \R^2 : |x_1| \leq a \text{ and } |x_2| \leq b\right\rbrace.
\end{equation*}
\begin{defn} \rm{(Minimal vectors)}
Let $\Lambda \in X$. A vector $\vr =(r_1, r_2) \in \Lambda$ is called a \textit{minimal vector} if it is nonzero, and if 
\begin{equation*}
    \vs = (s_1, s_2) \in R\left(|r_1|, |r_2|\right) \cap \Lambda\smallsetminus\{0\} \implies 
    |s_1| = |r_1| \text{ and } |s_2| = |r_2|.
\end{equation*}
\end{defn}
\begin{defn}\rm {(Consecutive minimal vectors)} 
An ordered pair of minimal vectors \linebreak $\vr = (r_1,r_2)$ and $\vs = (s_1, s_2)$ in $\Lambda$ is said to be the pair of \textit{consecutive minimal vectors} if $|r_2| < |s_2|$, and if there are no minimal vectors $\vw = (w_1,w_2)\in\Lambda$ with $|r_2| < |w_2| < |s_2|$.
\end{defn}

\begin{lem}[\cite{C2}, Lemma 2 and Proposition 3]\label{rectangle-interior}
If $\vr, \vs$ are consecutive minimal vectors in $\Lambda$ then we have
\begin{equation*}
    \operatorname{Interior}(R\left(|r_1|, |s_2|\right)) \cap \Lambda = \{0\}.
\end{equation*}
Moreover, {$\vr$ and $ \vs$} form a basis of $\Lambda$.
\end{lem}
If $x \in \R$, let $\lfloor x \rfloor$ and $\{x\}$ denote the integer and fractional parts of $x$. Consider the set
\begin{equation*}\label{U-def}
    \mathcal{U} = (0,1)^2 \cup \left( \left[0, \frac{1}{2} \right]\times \{0\}\right) \cup \left(\{0\}\times \left[0,\frac{1}{2}\right]\right).
\end{equation*}
\begin{equation*}\label{square}
    \begin{tikzpicture}[scale=1, baseline=(current  bounding  box.center)]


    \draw [thick, blue] (-2,-2) -- (0,-2);
    \draw [thick, blue, dotted] (0,-2) -- (2,-2);
    \draw [thick, blue, dotted] (2,-2) -- (2,2);
    \draw [thick, blue] (-2,-2) -- (-2,0);
    \draw [thick, blue, dotted] (-2,0) -- (-2,2);
    \draw [thick, blue, dotted] (-2,2) -- (2,2);
    
    \fill [red] (-2,-2) circle[radius=0.03];

    \node [left] at (-2,-2) {\tiny $(0,0)$};
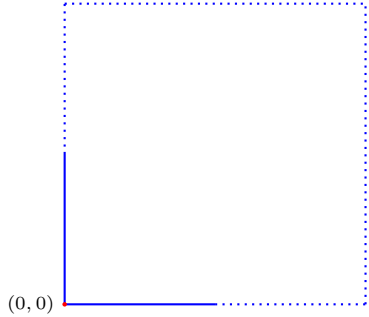
    \node [below] at (0,-2.2) {\tiny Here is a beautiful picture of $\mathcal{U}$.};
    \end{tikzpicture}
\end{equation*}
Consider also the maps
\begin{defn}\label{Gauss-map}\rm
$T:(0,1)^2 \cup \left( \left(0, \frac{1}{2}\right] \times \{0\}\right) \to  (0,1)^2 \cup \left(\{0\} \times \left(0, \frac{1}{2}\right]\right)$ given by
\begin{equation}
    T(x,y) = \left(\left\lbrace x^{-1}\right\rbrace, \frac{1}{\left\lfloor x^{-1} \right\rfloor + y} \right)
\end{equation}
and $S : (0,1)^2 \cup \left( \{0\} \times \left(0, \frac{1}{2}\right] \right) \to (0,1)^2 \cup \left(\left(0, \frac{1}{2}\right]\times \{0\}\right)$
given by 
\begin{equation*}
    S(a,b) = \left(\frac{1}{a+ \left\lfloor b^{-1} \right\rfloor}, \left\lbrace b^{-1} \right\rbrace\right).
\end{equation*}
\end{defn}
It is then straightforward to check that 
\begin{lem}
The compositions $S\circ T$ and $T \circ S$ give the identity on the domains of $T$ and $S$ respectively.
\end{lem}
We now discuss a mapping from pairs of minimal vectors of a lattice onto $\mathcal{U}\times \left\lbrace \pm 1 \right\rbrace$ that, in a sense made precise below, intertwines the diagonal action on $X$ with the $T$-action on $\mathcal{U}$.
\begin{equation*}\label{mapping}
    \begin{tikzpicture}[scale=1, baseline=(current  bounding  box.center)]
    
    \draw [<->](-1.5,0) -- (1.5,0);
    \draw [<->] (0,-0.5) -- (0,4);
    

    \draw [thick, blue, ->] (0,0) -- (1,0.4);
    \draw [thick, blue, ->] (0,0) -- (-0.3,1);
    \draw [thick, blue, ->] (0,0) -- (0.1, 3.4);
    
    \draw [thick, blue, ->, dotted] (1,0.4) -- (0.7,1.4);
    \draw [thick, blue, ->, dotted] (0.7, 1.4) -- (0.4, 2.4);
    \draw [thick, blue, ->, dotted] (0.4, 2.4) -- (0.1, 3.4);
    
    \fill [red] (0,0) circle[radius=0.03];

    \node [right] at (1,0.4) {\tiny $\vr$};
    \node [left] at (-0.3,1) {\tiny $\vs$};
    \node [right] at (0.1,3.4) {\tiny $\vw = \vr + 3\vs$};
    \node [right] at (0.7, 1.4) {\tiny $\vr + \vs$};
    \node [right] at (0.4, 2.4) {\tiny $\vr + 2\vs$};
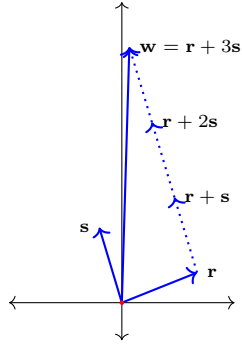
    \node [below] at (0,-0.5) {\tiny The Euclidean algorithm for finding the minimal vector following $\vs$.};
    \end{tikzpicture}
\end{equation*}
\begin{prop}[Proposition 6 in \cite{C2}]\label{map-to-transversal}
Let $\Lambda \in X$ and let $\vr, \vs$ be a pair of consecutive minimal vectors in $\Lambda$ with
\begin{equation*}
    0 \leq r_2 < s_2.
\end{equation*}
If we happen to have $r_2=0$ and $r_1 s_1 >0$, replace $\vr$ by $-\vr$.
For such pairs of minimal vectors, define the functions
\begin{equation*}
    x(\vr,\vs) := -\frac{s_1}{r_1},\ y(\vr,\vs) := \frac{r_2}{s_2},\ \varepsilon(\vr,\vs) := \frac{r_1}{|r_1|}.
\end{equation*} 
We then have that 
    $(x,y) \in \mathcal{U}.$
Moreover, if $s_1 \neq 0$, $n := \left\lfloor \frac{1}{x}\right\rfloor$ and
    $\vw := \vr + n\vs,$
then $\vs, \vw$ is a pair of consecutive minimal vectors for $\Lambda$.
Further, we have that 
\begin{equation}\label{coord-relation}
\big(x(\vs,\vw) ,y(\vs,\vw), \varepsilon(\vs, \vw)\big) = \big(T\left(x(\vr,\vs), y(\vr,\vs)\big), - \varepsilon(\vr,\vs)  \right),
\end{equation}
where $T$ is as in Definition \ref{Gauss-map}.
\end{prop}
\begin{rem} \rm
Note that $r_1$ cannot be zero by minimal vector considerations. 
In subsequent notation, we often drop the dependence of $x, y, \varepsilon$ on the minimal vectors when the context is clear.
We repeat 
{the proof from  \cite{C2}} in order to make {use} of equation \eqref{intrinsic-basis} {below}.
\end{rem}
\begin{proof} [Proof of Proposition \ref{map-to-transversal}]
We first show that $(x,y) \in \mathcal{U}$. It is clear that $0 \leq y < 1$. 

If $y=0$, then we have $r_2 = 0$ and, by definition, either $s_1=0$ or $r_1 s_1 < 0$. This shows that $0\leq x$.
Moreover, if $s_1 =0$, then $(x,y) = (0,0) \in \mathcal{U}$ as desired.
So, assume we are in the case when $s_2=0$ and $r_1 s_1 <0$.
Consider the lattice vector $\vr + \vs = (r_1 + s_1, s_2)$. Since $\vs$ is minimal, we must have \begin{equation}\label{y=0}
    |r_1 + s_1| \geq |s_1|.
\end{equation}
When $r_1 > 0$, since $|s_1| < r_1$ and $r_1s_1 <0$, this implies 
\begin{equation*}
    r_1 + s_1 \geq -s_1
\end{equation*}
which shows that $x \leq 1/2$. When $r_1 < 0$, equation \eqref{y=0} leads to 
\begin{equation*}
    -(r_1 + s_1) \geq s_1,
\end{equation*}
which again shows that $x\leq 1/2$.

Now consider $x$. By minimality, $|s_1| < |r_1|$ so that $|x| < 1$. If $x=0$ so that $s_1=0$, we consider the lattice vector $\vr - \vs = (r_1, r_2 - s_2)$. 
Since $\vr$ is minimal, we must have
\begin{equation*}
    |r_2 - s_2| = s_2 - r_2 \geq r_2
\end{equation*}
which leads to $y \leq 1/2$.
Taking stock so far, we have shown that $(x,y)$ always lies in 
\begin{equation*}
    [0,1/2]\times \{0\} \cup (0,1) \times (0,1) \cup (-1,0) \times (0,1) \cup \{0\} \times [0,1/2].
\end{equation*}

Thus, we are left with proving that $x > 0$ when $y>0$.
If $y>0$, consider the vector $\vr - \vs$ which is equal to $(r_1 - s_1, r_2 - s_2)$. If $r_1$ and $s_1$ had the same sign, this vector would belong to the interior of the rectangle $B(|r_1|, |s_2|)$ which contradicts Lemma \ref{rectangle-interior}. Thus, $x>0$ and this concludes the proof that $(x,y) \in \mathcal{U}$.

Moreover, we can write $\vr, \vs$ in a matrix as
\begin{equation*}
 \left[ {\begin{array}{cc} \vr & \vs \end{array}}\right] =
 \left[ {\begin{array}{cc} \varepsilon |r_1| &  - \varepsilon|r_1|x \\ s_2 y & s_2 \end{array}}\right].
\end{equation*}
Or rather, noting that $\Lambda$ is a covolume one lattice, we can solve for $s_2$ and write:
\begin{equation}\label{intrinsic-basis}
\left[ {\begin{array}{cc} \vr & \vs \end{array}}\right] =
    \left[ {\begin{array}{cc}  |r_1| &  0 \\ 0 & |r_1|^{-1} \end{array}}\right]
    \left[ {\begin{array}{cc} \varepsilon &  - \varepsilon x \\ \frac{y}{1+xy} & \frac{1}{1+xy} \end{array}}\right].
\end{equation}

Now we study the case when $s_1 \neq 0$. In this case, Minkowski's convex body theorem shows that there is a nonzero lattice vector in the strip
\begin{equation*}
    \left\lbrace (x_1,x_2) \in \R^2 : x_2 >0 \text{ and } |x_1| < |s_1|\right\rbrace.
\end{equation*}
But using the discreteness of $\Lambda$, we get a minimal lattice vector $\vw = (w_1, w_2)$ in this region such that $\vs, \vw$ form a consecutive pair of minimal vectors.
By Lemma \ref{rectangle-interior}, we see that $\vr, \vs$ and $\vs, \vw$ form bases of $\Lambda$. The orientations of each basis are distinct by what we have proved above; namely, if $r_1>0$, then $\det(\vr,\vs)= 1$, and if $r_1<0$, then $\det(\vr, \vs)=-1$. And further, we always have the condition that $r_1 s_1 <0$.
Thus we can write
\begin{equation*}
    \vw = \vr + k\vs = \left(\varepsilon|r_1|x\left(\frac{1}{x}-k\right), s_2(y+k)\right) \text{ for some } k \in \mathbb{N}.
\end{equation*}
(Note that $x \neq 0$ since $s_1 \neq 0$.)
Since $\vw$ is minimal, we must have 
\begin{equation*}
    \left| \frac{1}{x} - k\right| < 1,
\end{equation*}
so that $k= n$ or $n+1$.
On the other hand the presence of the lattice vector 
\begin{equation*}
    \vr + n\vs = \left(\varepsilon |r_1| x\Big(\frac{1}{x}-n\Big), s_2(y+n)\right),
\end{equation*}
and the fact that $\vs, \vw$ are consecutive shows that we indeed have $\vw = \vr + n\vs$. 
{Also} we have
\begin{equation*}
    x(\vs,\vw) = -\frac{w_1}{s_1} = \left\lbrace \frac{1}{x(\vr,\vs)}\right\rbrace,\ y(\vs, \vw) = \frac{s_2}{w_2} = \frac{1}{y(\vr,\vs) + n},\ \varepsilon(\vs, \vw) = -\varepsilon(\vr, \vs),
\end{equation*}
thereby completing the proof.
\end{proof}
\begin{rem} \rm
Henceforth, whenever we refer to a minimal vector $\vr = (r_1, r_2)$, we always make the tacit assumption that $r_2\geq 0$.
\end{rem}
{We also note the following characterization of precompact orbits:}
\begin{prop}\label{bounded-bad}
{The orbit $g_\R \Lambda$ is precompact} if and only if, for {any} choice of consecutive minimal vectors $\vr, \vs\in\Lambda$, we have that both $x(\vr,\vs)$ and $y(\vr,\vs)$ are badly approximable real numbers.
\end{prop}
\begin{proof}
Say $g_\R\Lambda$ is precompact in $X$. Let $\vr=(r_1, r_2), \vs= (s_1, s_2)$ be any pair of consecutive minimal vectors.
By Lemma \ref{rectangle-interior}, they form a basis of $\Lambda$ and we can write
\begin{equation*}
    \Lambda =  \left[ {\begin{array}{cc} r_1 &  s_1 \\ r_2  & s_2 \end{array}}\right]\Z^2.
\end{equation*}
Moreover, {precompactness shows that $\Lambda$ does not contain any vector on the coordinate axes}, and so we have
\begin{equation*}
    \left[ {\begin{array}{cc} r_1 &  s_1 \\ r_2  & s_2 \end{array}}\right]
    = 
    \left[ {\begin{array}{cc} r_1 &  0 \\ r_2  & s_2 - \frac{s_1r_2}{r_1} \end{array}}\right]
    \left[ {\begin{array}{cc} 1 &  \frac{s_1}{r_1} \\ 0  & 1 \end{array}}\right] = 
    \left[ {\begin{array}{cc} r_1 &  0 \\ r_2  & s_2 - \frac{s_1r_2}{r_1} \end{array}}\right]
    u_{-x(\vr,\vs)},
\end{equation*}
so that, from the commutator relation 
\eqref{commutator} and the equivalence 
{of} Proposition \ref{dynamical-restatement}(i), $-x(\vr,\vs)$ is badly approximable,
which, of course, is equivalent to $x(\vr,\vs)$ being badly approximable. 

Similarly, we also have the decomposition
\begin{equation*}
        \left[ {\begin{array}{cc} r_1 &  s_1 \\ r_2  & s_2 \end{array}}\right]
    = 
    \left[ {\begin{array}{cc} r_1 - \frac{s_1r_2}{s_2} &  s_1 \\ 0  & s_2 \end{array}}\right]
    \left[ {\begin{array}{cc} 1 &  0 \\ \frac{r_2}{s_2}  & 1 \end{array}}\right] = 
    \left[ {\begin{array}{cc} r_1 - \frac{s_1r_2}{s_2} &  s_1 \\ 0  & s_2 \end{array}}\right]
    v_{y(\vr,\vs)}.
\end{equation*}
Using \eqref{commutator} and {\equ{backwards}}, we see that $y(\vr,\vs)$ is badly approximable.
We remark that, in the above proof, it does not matter if the matrix formed by the ordered pair $\vr,\vs$ has determinant $-1$.

For the converse, we simply note that we can reverse the order in the argument just given.
\end{proof}

\section{{An abundance of precompact orbits with a prescribed limit point}}\label{hd-computation}

We set out some further notation involving continued fractions.
The symbols 
\begin{equation*}
    a_1, a_2, \dots, a_n,\dots
\end{equation*} denote the countable family of measurable functions giving the continued fraction {coefficients} of a number in $[0,1]$. 
Of course, $a_n$ is not defined on certain rationals. 
Other letters such as $b_n, c_n, d_n$ are, as before, used to denote continued fraction {coefficients} of specific numbers. 
We also have the associated convergent functions
\begin{equation*}
    \frac{p_n}{q_n}:= [0;a_1,\dots,a_n],
\end{equation*}
and the remainder functions defined by
\begin{equation*}
    \alpha = [0; a_1(\alpha), \dots, a_{n-1}(\alpha), \rho_n(\alpha)].
\end{equation*}
Finally, we denote the the successive iterates of the Gauss map on $[0,1)$ by
\begin{equation*}\label{z-notation}
    z_n(\alpha):= \rho_n(\alpha) - a_n(\alpha) \in [0,1).
\end{equation*}
Given a multi-index $\vk = (k_1, \dots, k_n)$ of natural numbers, we define the $n$th order cylindrical interval
\begin{equation*}
    I(\vk) := \left\lbrace \alpha \in [0,1]: a_i(\alpha) = k_i \text{ for } i = 1,\dots, n\right\rbrace.
\end{equation*}
It is well known that these intervals are given by 
\begin{equation}\label{cylinder}
I(\vk)=\begin{cases}
\left[\frac{p_n}{q_n}, \frac{p_n + p_{n-1}}{q_n + q_{n-1}}\right) & \text{ if $n$ is even,} \\
\left(\frac{p_n + p_{n-1}}{q_n + q_{n-1}}, \frac{p_n}{q_n}\right] & \text{ if  $n$ is odd}.
\end{cases}
\end{equation}
Here, the convergents are associated to $\vk$.
The following estimate on conditional probability is essential for our computation of dimension. Here and hereafter, $|\cdot|$ denotes Lebesgue measure.
\begin{prop}[cf.\ \cite{IK}, Corollary 1.2.6]\label{conditional}
If $\vk = (k_1,\dots, k_n)$ is a multi-index of natural numbers and $x \in [0,1]$, then we have
\begin{equation*}\label{conditional-eqn}
    \frac{\left|\left\lbrace \alpha \in I(\vk): z_{n}(\alpha) \geq x  \right\rbrace \right|}{\left|I(\vk)\right|} = \frac{1-x}{\sigma_n x + 1}
\end{equation*}
where $\sigma_n:= q_{n-1}q_n^{-1}$  {($\sigma_n$ is dependent on $\vk$ only)}.
\end{prop}
\ignore{\begin{proof}\comm{I am confused: why are you proving this if the statement came from \cite{IK}?}
We estimate the measure of the set of $\alpha \in [0,1]$ 
{such that}
\begin{equation*}
    \alpha = \frac{r p_n + p_{n-1}}{r q_n + q_{n-1}} \text{ with } 1 < r < \infty  \text{ and } r \leq \frac{1}{x}.
\end{equation*}
(We assume that $x\neq 0$ since the result is trivial in that case.)
Upto a set of measure zero, this is exactly the set of $\alpha\in [0,1]$ for which
\begin{equation*}
    \alpha \in I(\vk),\ r_{n+1}(\alpha) = r,\ \text{ with }\ z_n(\alpha) := r_n(\alpha) - a_n(\alpha) = \frac{1}{r} \geq x.
\end{equation*}
Thus, according to \eqref{cylinder} and the condition $1 < r \leq x^{-1}$, we compute
\begin{equation*}
\begin{split}
    \left| \frac{p_n + p_{n-1}}{q_n + q_{n-1}} - \frac{x^{-1}p_n + p_{n-1}}{x^{-1}q_n + q_{n-1}}\right| &= \left| \frac{(-1)^{n+1} + (-1)^nx^{-1}}{(q_n + q_{n-1})(x^{-1}q_n+q_{n-1})}\right| \\
    &= \frac{1-x}{(q_n + q_{n-1})q_n(1 + xs_n)}.
    \end{split}
\end{equation*}
Dividing by $|I(\vk)|$ gives the result.
\end{proof}}
As a corollary, we obtain
\begin{cor}\label{cor-conditional}
If $0 < x < y<1$, we have
\begin{equation*}
 \frac{\left|\{\alpha \in I(\vk): x \leq z_{n}(\alpha) \leq y\}\right|}{|I(\vk)|} = 
 \frac{(y-x)(\sigma_n +1)}{(\sigma_n x+1)(\sigma_n y +1)}.
\end{equation*}
\end{cor}
We also have the following 
{consequence} of the mass distribution principle, 
{see} \cite[Lemma 2.1]{U}.
For each $m\in \N\cup \{0\}$ let $\mathcal{E}_m \subset 2^{[0,1]}$ be a finite collection of nondegenerate compact intervals in $[0,1]$.
By abuse of notation we write 
\begin{equation*}
    \cup \mathcal{E}_m := \bigcup_{I \in \mathcal{E}_m} I.
\end{equation*}
Define for each $I \in \mathcal{E}_m$, 
\begin{equation*}
    \text{density}(\mathcal{E}_{m+1}, I) := \frac{|\cup \mathcal{E}_{m+1} \cap I|}{|I|}.
\end{equation*}
Assume we have the following properties:
    \begin{enumerate}
        \item[(a)] 
        $\cup \mathcal{E}_0 = [0,1]$.
        \item[(b)] For $I \neq J$ in $\mathcal{E}_m$, we have $|I\cap J| = 0$.
        \item[(c)] For every $I \in \mathcal{E}_{m+1}$, there is a unique $J \in \mathcal{E}_m$ with $I \subset J$.
        \item[(d)] We have that        \begin{equation*}\label{Cantor-construction}
            \Theta_m := \inf \left\lbrace \text{density}(\mathcal{E}_{m+1},I) : I \in \mathcal{E}_m \right\rbrace > 0.
        \end{equation*}
        \item[(e)] If $\text{diam}_m$ is the supremum of $\text{diameter}(I)$ over all $I \in \mathcal{E}_m$, then we have
        \begin{equation*}
            \lim_{m\to \infty} \text{diam}_m = 0.
        \end{equation*}
    \end{enumerate}
    \begin{thm}[Theorem 2.1, \cite{U}]\label{Urbanski}
Given the conditions above, let ${E} = \bigcap_{m \in \N} \cup \mathcal{E}_m$. Then we have that
\begin{equation*}
    1- \dim{{E}} \leq \limsup_{m \to \infty} \frac{\sum_{j=1}^{m-1} \log \Theta_j}{\log \operatorname{diam}_m}.
\end{equation*}
\end{thm}

\begin{proof}[Proof of Theorem \ref{main-lemma}] 
We have that $g_\R \Lambda$ is precompact. Choose consecutive minimal vectors $\vr, \vs \in \Lambda$.
By Proposition \ref{bounded-bad}, we have that $x= x(\vr, \vs)$ and $y= y(\vr, \vs)$ are badly approximable. 
From Proposition \ref{map-to-transversal}, we can assume without loss of generality, that $\varepsilon = \varepsilon(\vr, \vs) = 1$.
From equation \eqref{intrinsic-basis}, $\vr, \vs$ have coordinates given by
\begin{equation*}
    \left[ {\begin{array}{cc}  \vr &  \vs  \end{array}}\right] = 
    \left[ {\begin{array}{cc}  |r_1| &  0 \\ 0 & |r_1|^{-1} \end{array}}\right]
    \left[ {\begin{array}{cc} 1 &  -  x \\ \frac{y}{1+xy} & \frac{1}{1+xy} \end{array}}\right].
\end{equation*}
And by Proposition \ref{rectangle-interior},  {$\vr$ and $ \vs$ form a basis of}
$\Lambda$.

Write out the continued fraction expansions
\begin{equation*}
    x = [0;b_1, \dots, b_m, \dots] \text{ and } y = [0; c_1, \dots, c_m, \dots]
\end{equation*}
and,  {recalling that $x, y$ are badly approximable, let} $M \in \R$ be {a} uniform upper bound for $b_m, c_m$.
For each $k \in \N$, let $B_k$ be the block of $2k$ digits
\begin{equation*}
    B_k = (c_k, \dots, c_1, b_1, \dots, b_k). 
\end{equation*}
Let $\alpha \in (0,1)$ be any real number which has infinite continued fraction expansion
\begin{equation*}
    \alpha = [0; d_1, \dots, d_m, \dots].
\end{equation*}
Assume that, for each $k\in \N$, there is a digit $d_{m_k}$ with $m_k$ {odd} and such that
\begin{equation}\label{parity}
    \left(d_{m_k-(k-1)}\dots,d_{m_k}, d_{m_k + 1}\dots, d_{m_k + k}\right) = B_k.
\end{equation}
Moreover, assume that 
    $d_1 = 2$.
{We make this assumption for convenience in writing out the minimal vectors of $\Lambda_\alpha$. In any case, it will lead us to a set of full dimension.}

We then claim that
    $\Lambda \in \overline{g_{\R_{>0}}\Lambda_\alpha}$.
To prove the claim, it suffices to show that
\begin{equation}\label{orbit-claim}
    g_{-\log(|r_1|)}\Lambda = \left[ {\begin{array}{cc} 1 &  -  x \\ \frac{y}{1+xy} & \frac{1}{1+xy} \end{array}}\right] \Z^2 \text{ belongs to } \overline{g_{\R_{>0}}\Lambda_\alpha}.
\end{equation}

Consider the {following} consecutive minimal vectors in $\Lambda_\alpha$:
\begin{equation*}
    \vw_1 = \left[ {\begin{array}{c} -1 \\ 0 \end{array}}\right]
    \text{ and } 
    \vw_2 = \left[ {\begin{array}{c} \alpha \\ 1 \end{array}}\right].
\end{equation*}
Here we have used that the first digit of $\alpha$ is $2$.
According to Proposition \ref{map-to-transversal}, the {associated} coordinates of $\vw_1, \vw_2$ are given by
\begin{equation}\label{eq: base case coordinates}
     x(\vw_1, \vw_2) = \alpha,\  y(\vw_1, \vw_2) = 0, \text{ and }  \varepsilon(\vw_1, \vw_2) = -1.
\end{equation}
{We claim that, by induction,} it is possible to continue choosing minimal vectors and obtain a sequence $(\vw_n)_{n \in \N}\subset \Lambda_\alpha$ such that $\vw_n, \vw_{n+1}$ form consecutive pairs and that
\begin{equation}\label{iterated-formulas}
    x(\vw_n, \vw_{n+1}) = [0;d_{n},d_{n+1}, \dots],\ 
    y(\vw_{n}, \vw_{n+1}) = [0; d_{n-1}, \dots, d_1],\ \varepsilon(\vw_n, \vw_{n+1}) = 
    (-1)^{n}.
\end{equation}
{The convention for $n=1$ is that $y(\vw_{1}, \vw_{2})=0$.
The base case of \eqref{iterated-formulas} is nothing but \eqref{eq: base case coordinates}.
For the induction step, assume that \eqref{iterated-formulas} holds. Then note that, since $\alpha$ is irrational, $\Lambda_\alpha$ does not contain any vector on the $y$-axis. Thus, given the consecutive pair of minimal vectors $\vw_{n}, \vw_{n+1}$ for $\Lambda_\alpha$, the latter part of Proposition \ref{map-to-transversal} applies. This gives us the vector $\vw_{n+2}$, and formula \eqref{coord-relation} says that}
\begin{equation*}
    {x(\vw_{n+1}, \vw_{n+2}) =  \left\lbrace [0;d_n,d_{n+1}\dots]^{-1}\right\rbrace = [0;d_{n+1}, d_{n+1},\dots]},
\end{equation*}
\begin{equation*}
    {y(\vw_{n+1}, \vw_{n+2}) =  \frac{1}{\lfloor [0;d_n,d_{n+1},\dots]^{-1} \rfloor + [0; d_{n-1}, \dots, d_1]} = [0;d_n,\dots,d_1] }
\end{equation*}
and 
\begin{equation*}
   { \varepsilon(\vw_{n+1}, \vw_{n+2}) = - (-1)^n = (-1)^{n+1}.}
\end{equation*}
{Thus, formula \eqref{iterated-formulas} is proved and we write $x_n, y_n, \varepsilon_n$ for the {respective} coordinates.} Again using \eqref{intrinsic-basis} and the fact that $\vw_n, \vw_{n+1}$ form a basis of $\Lambda_\alpha$, we write, for some $t_n >0$, that
\begin{equation*}
    \Lambda_\alpha = \left[ {\begin{array}{cc} e^{-t_n} & 0 \\ 0 & e^{t_n} \end{array}}\right]
    \left[ {\begin{array}{cc} \varepsilon_n & -\varepsilon_n x_n \\ \frac{y_n}{1+ x_ny_n} & \frac{1}{1+ x_n y_n} \end{array}}\right] \Z^2.
\end{equation*}
Or rather,
\begin{equation*}
    g_{t_n} \Lambda_\alpha = \left[ {\begin{array}{cc} \varepsilon_n & -\varepsilon_n x_n \\ \frac{y_n}{1+ x_ny_n} & \frac{1}{1+ x_n y_n} \end{array}}\right] \Z^2.
\end{equation*}
If we choose {$n = m_k + 1$}, the choices in \eqref{parity} and the formulas in \eqref{iterated-formulas} show that
\begin{equation*}
    x_{n} = [0;d_{m_k +1}, \dots] = [0; b_1, \dots, b_k, d_{m_k + k +1}, \dots],
\end{equation*}
\begin{equation*}
    y_{n} = [0;d_{m_k}, \dots, d_{m_k - (k-1)}, \dots, d_1] = [0; c_1, \dots, c_k, d_{m_k - k}, \dots , d_1]
\end{equation*}
and that 
\begin{equation*}
    \varepsilon_{n} = (-1)^{m_k +1} = 1. 
\end{equation*}
Hence the matrices 
\begin{equation*}
    \left[ {\begin{array}{cc} 1 &  -  x \\ \frac{y}{1+xy} & \frac{1}{1+xy} \end{array}}\right] 
    \text{ and } 
    \left[ {\begin{array}{cc} \varepsilon_n & -\varepsilon_n x_n \\ \frac{y_n}{1+ x_ny_n} & \frac{1}{1+ x_n y_n} \end{array}}\right]
\end{equation*}
can be made arbitrarily close, which establishes the claim in \eqref{orbit-claim}.
Thus we have shown that $\mathcal{S}_\Lambda$ is nonempty, and {it remains to estimate its} Hausdorff dimension.

We will work with a convenient subset of $\mathcal{S}_\Lambda$.
What we have so far is that the lattice $\Lambda$ with precompact orbit gives sequences of digits $(b_n), (c_n)$ and a bound $M \in \N$ from which we get the blocks
\begin{equation*}
    B_k = \left(c_{k}, \dots, c_1, b_1,\dots, b_k\right)\ \text{ with } 1 \leq c_j, b_j \leq M.
\end{equation*}
The set we are concerned with is
\begin{equation*}
    \widetilde{\mathcal{S}}_\Lambda := \bigcap_k \bigcup_{m \text{ is {odd}}} \left\lbrace \alpha \in \BA :a_1(\alpha)=2 \text{ and } \big(a_i(\alpha)\big)^{m+k}_{m-(k-1)} = B_k\right\rbrace,
\end{equation*}
which we have just shown to be a subset of $\mathcal{S}_\Lambda$.
We are left to show $\dim \widetilde{\mathcal{S}}_\Lambda = 1$.
Let $L > M$ be a natural number. Choose a sequence of {odd} numbers $(m_k)_{k \in \N}$ sparse enough so that
\begin{enumerate}
    \item[(a)] For each $k>1$, $m_{k}-m_{k-1} > 2k$.
    \item[(b)] If we define
    \begin{equation}\label{sparse}
        \text{density}(\big(m_k), n\big):= \frac{\sum_{m_k<n} 2k}{n},\ \text{ then }\lim_{n\to \infty}\text{density}((m_k),n)=0.
    \end{equation}
\end{enumerate}
We define the {following} subset of $\widetilde{\mathcal{S}}_\Lambda$:
\begin{equation*}
    \mathcal{S}\big(L,(m_k)\big) := \bigcap_k \left\lbrace \alpha \in \tilde{\mathcal{S}}_\Lambda: a_n(\alpha) \leq L \text{ for each }n,\text{ and } \big(a_n(\alpha)\big)^{m_k+k}_{m_k-k+1} = B_k \right\rbrace.
\end{equation*}
This is nothing but the set of badly approximable numbers with $a_1=2$  having upper bound $L$ on their 
{partial quotients} and having the block $B_k$ appear at the $m_k$ position.
The upper bound on the entries of $B_k$ and condition (a) above guarantee that this set is nonempty.

We now use the Cantor set structure of $\mathcal{S}\big(L,(m_k)\big)$ to estimate its Hausdorff dimension.
According to the notation in Theorem \ref{Urbanski}, we first define
$$
\mathcal{E}_0 := \{[0,1]\}\text{ and } \mathcal{E}_1 := \left\lbrace \overline{\{\alpha\in [0,1]: a_1(\alpha) =2\}} \right\rbrace   {= \left\lbrace  [1/3,1/2] \right\rbrace}.
$$
{Now, assuming we have defined $\mathcal{E}_{m}$ as a collection of closed intervals $\left\lbrace\overline{I}\right\rbrace$ where each  $I$ is of the form \eqref{cylinder} (with $n=m$), we define the family $\mathcal{E}_{m+1}$ as follows:} $\mathcal{E}_{m+1}:=$
\begin{equation}\label{Em-definition}
    \begin{cases}
    \bigcup\limits_{I \in \mathcal{E}_{m}} \bigcup\limits_{i\leq L} \left\lbrace\overline{\{\alpha \in I: a_{m+1}(\alpha)=i \}} \right\rbrace & \text{if for all } k,\ m+1 \notin \{m_k-(k-1),\dots,m_k +k\},  \\
    \bigcup\limits_{I \in \mathcal{E}_{m}} \left\lbrace \overline{\{\alpha\in I: a_{m+1}(\alpha) = c_{i+1}\}} \right\rbrace &  \text{if } m+1=m_k-i \text{ where } 0 \leq i \leq k-1, \\
    \bigcup\limits_{I \in \mathcal{E}_{m}} \left\lbrace \overline{\{\alpha\in I: a_{m+1}(\alpha) = b_{i}\}} \right\rbrace & \text{if } m+1=m_k+i \text{ where } 1 \leq i \leq k.
    \end{cases}
\end{equation}

The cylindrical intervals of \eqref{cylinder} are either disjoint or satisfy a containment relation.
{Thus each $I, J$  which belong to some $\mathcal{E}_m$ can intersect in at most a point.} The diameter of these sets converge to $0$ as $m \to \infty$. Moreover for each $I \in \mathcal{E}_m$, we have 
\begin{equation*}
    |\cup \mathcal{E}_{m+1} \cap I| >0.
\end{equation*}
Thus the conditions for Theorem \ref{Urbanski} are satisfied.
We also have that
\begin{equation*}
    \mathcal{S}\big(L,(m_k)\big) = \bigcap_m \cup \mathcal{E}_m.
\end{equation*}

We now turn to computing $\Theta_m$. Assume $\mathcal{E}_{m+1}$ is defined according to the first case in equation \eqref{Em-definition}. If $I \in \mathcal{E}_m$, then, using Proposition \ref{conditional}, we can wirte
\begin{equation*}
\begin{split}
    |\cup \mathcal{E}_{m+1} \cap I| &= \left|\left\lbrace \alpha \in I : a_{m+1}(\alpha) \leq L \right\rbrace\right| \\
    &= \left|\left\lbrace \alpha \in I: z_{m}(\alpha)\geq (L+1)^{-1} \right\rbrace \right| \\
    &= \frac{1- (L+1)^{-1}}{\sigma_{m}(L+1)^{-1}+1} \cdot |I|.
    \end{split}
\end{equation*}
Here, $\sigma_{m}$ depends on the defining {continued fraction coefficients} for $I$.
But since $\sigma_{m}<1$ always, we get that
\begin{equation}\label{Del1}
    \Theta_m \geq \frac{L}{L+2}.
\end{equation}
Now if $\mathcal{E}_{m+1}$ is defined according to the second or third case in equation \eqref{Em-definition}, using Corollary \ref{cor-conditional} and the bound $M$, we have
\begin{equation}\label{Del2}
    \Theta_m \geq \frac{1}{4(M+1)^2}.
\end{equation}
In order to differentiate these two ways in which $\mathcal{E}_{m+1}$ can be defined, we introduce, for each $m\in \N$, the index sets
\begin{equation*}
    \mathcal{I}_m = \bigcap\limits_{k=1}^{\infty} \left\lbrace j \in \N_{< m} : j+1 \notin \{m_k-(k-1),\dots, m_k+k\}\right\rbrace
\end{equation*}
and
\begin{equation*}
    \mathcal{J}_m = \bigcup\limits_{k=1}^{\infty} \left\lbrace j \in \N_{< m} : j+1 \in \{m_k-(k-1),\dots, m_k+k\}\right\rbrace.
\end{equation*}

In considering $\text{diam}_m$, we note by induction and Corollary \ref{cor-conditional} that
\begin{equation}\label{dm}
    \text{diam}_m \leq 2^{-m}.
\end{equation}
Thus, we can finally estimate, using \eqref{Del1}, \eqref{Del2}, \eqref{dm},
\begin{equation*}
    \begin{split}
        \frac{-\sum_{j=1}^{m-1} \log \Theta_j}{-\log \text{diam}_m} & = \frac{1}{-\log \text{diam}_m}\left(-\sum\limits_{j\in \mathcal{I}_m}\log \Theta_j - \sum\limits_{j\in \mathcal{J}_m}  \log \Theta_j\right) \\
        &\leq \frac{1}{m\log 2} \left(-\sum\limits_{j\in \mathcal{I}_m} \log \frac{L}{L+2} - \sum\limits_{j\in \mathcal{J}_m} \log \frac{1}{4(M+1)^2} \right) \\
        &= \frac{-1}{\log 2} \left(\frac{\# \mathcal{I}_m}{m} \log \frac{L}{L+2} + \frac{\#\mathcal{J}_m}{m} \log \frac{1}{4(M+1)^2}\right).
     \end{split}
\end{equation*}
Using Theorem \ref{Urbanski} and 
\eqref{sparse} on the above estimate, we get
\begin{equation*}
    1 - \dim  \mathcal{S}\big(L,(m_k)\big) \leq \limsup_{m\to \infty } \frac{-\sum_{j=1}^{m-1} \log \Theta_j}{-\log \text{diam}_m} \leq 
    \frac{1}{\log 2} \cdot \log \frac{L+2}{L}.
\end{equation*}
Since $ \mathcal{S}\big(L,(m_k)\big)\subset \mathcal{S}_\Lambda$, taking $L \to \infty$ gives the result. 
\end{proof}
\section{Further questions on Dirichlet improvability and precompact orbits}

Some amusing problems on $\nu$-Dirichlet numbers which we omitted are:
\begin{enumerate}
    \item[(c)] In the case of the Euclidean norm, does $\DI_2 \smallsetminus \BA$ have full Hausdorff dimension?
    \item[(d)] In general, for two norms $\nu_1, \nu_2$ on $\R^2$, what can be said about $\DI_{\nu_1} \smallsetminus \DI_{\nu_2}$?
\end{enumerate}
We expect (c) to have an affirmative answer, although constructing the required orbits in $X$ will require a more delicate inductive procedure than the {one} produced here.

\smallskip
Dirichlet-improvability can be studied in a variety of different settings. And each setting relates, via the Dani correspondence, to diagonal orbits on a homogeneous space avoiding a critical locus.
The setting of diagonal flows on $\SL_d(\R)/\SL_d(\Z)$, described in \cite{KR2}, is as follows. Fix positive integers $m, n$ with $m+n = d$ and a set of weights 
\begin{equation*}
    {\vom} = (\val, \vbe) \in \R^m_+ \times \R^n_+ \text{ with } \sum_i \alpha_i = \sum_j \beta_j= 1.
\end{equation*}
Fix also a norm $\nu$ on $\R^d$ and define the critical radius $r_\nu$ in analogy with \eqref{critical-radius}. A matrix $A\in \mr(\R)$ is said to be $(\nu,\vom)$-Dirichlet improvable (written $A\in \DI_{\nu,\vom}$) if there is a constant $c<1$ such that 
\begin{equation*}
    \Lambda_A \cap \left[ {\begin{array}{cc} \left(ct^{-1}\right)^\val &  0 \\ 0 & (t)^\vbe \end{array}}\right] B_\nu(r_\nu) \neq \{0\}
\end{equation*}
for all $t$ sufficiently large. 
Here, we define a positive real number $x$ raised to a vector power $\va \in \R^k$ as the $k\times k$ matrix
\begin{equation*}
    (x)^\va := \operatorname{diag}(x^{a_1}, \dots, x^{a_k}),
\end{equation*}
and 
use the notation
\begin{equation*}
    \Lambda_A : =  \left[ {\begin{array}{cc} I_m &  A \\ 0 & I_n \end{array}}\right]\Z^d.
\end{equation*}
We also have the set of \textit{weighted badly approximable} matrices:
\begin{equation*}\label{wBA} \BA_\vom := \left\{\text{\amr}(\R):   \inf_{\vp \in \Z^m,\, \vq \in \Z^n\smallsetminus \{{\bf 0}\}} \|A\vq -\vp \|_\val  \|\vq\|_\vbe > 0 \right\}
\end{equation*}
where we have used the quasi-norms
    \begin{equation*}
\|\x\|_{\val}:=\max_i|x_i|^{1/\alpha_i} \quad\textrm{and}\quad \|\vy\|_{\vbe}:=\max_j|y_j|^{1/\beta_j}.
\end{equation*}
Now one can naturally ask for a comparison study between $\DI_{\nu, \vom}$ and $\BA_\vom$.
We 
{predict}  that the same dichotomy phenomenon as in Theorem \ref{dichotomy-thm} holds for this case.
(See \cite{KR2} for some preliminary theorems about these sets.)

An approach to this problem via symbolic dynamics is currently unavailable. However, one can try and modify the construction (via equidistibution of expanding horocycles) in \cite{KM} to force the resulting precompact orbits to have prescribed limit points. (Of course, it is necessary that these limit points themselves have precompact $\R$-orbits.)
We mention that, although this technique via equidistribution was available to us, we chose to use continued fractions in the case of $\SL_2(\R)/\SL_2(\Z)$ for the simplicity of their heuristics and as an exercise in understanding Poincare sections.

\subsection*{Acknowledgements}
The authors thank Barak Weiss for helpful  questions and comments  on an earlier version of the paper. The second-named author thanks S.\,G.\ Dani and Anish Ghosh for their patience and helpful discussions.
{We also thank the reviewer for reading the paper and making several suggestions that improved the writing.}

\end{document}